\documentclass[10pt,a4paper]{article}
\usepackage{amsmath}
\usepackage{amssymb}
\usepackage{amsthm}
\usepackage{amsfonts}
\usepackage{enumerate}
\usepackage{pstricks}
\usepackage{pst-plot}
\usepackage{pst-poly}
\usepackage{graphics} 
\usepackage{graphicx}
\usepackage{pgfpages}
\usepackage{url}
\usepackage{booktabs}
\usepackage{multirow}
\usepackage{verbatim}
\usepackage{hyperref}
\usepackage{breakurl}
\usepackage[margin=1.3in]{geometry}
\newcommand{\tluste}[1]{\mbox{\mathversion{bold}$ #1 $}}
\newcommand{\vr}[1]{{{#1}}}

\newcommand{\mace}[1]{{{#1}}}
\newcommand{\mna}[1]{{\mathcal{#1}}}

\newcommand{\omace}[1]{\mbox{$\overline{\mace{#1}}$}} 
\newcommand{\umace}[1]{\mbox{$\underline{\mace{#1}}$}} 
\newcommand{\imace}[1]{\mbox{$\tluste{#1}$}}

\def\Mid#1{#1^c}
\def\Rad#1{#1^\Delta}
\newcommand{\ovr}[1]{\mbox{$\overline{\vr{#1}}$}} 
\newcommand{\uvr}[1]{\mbox{$\underline{\vr{#1}}$}}

\newcommand{\onum}[1]{\mbox{$\overline{{#1}}$}} 
\newcommand{\unum}[1]{\mbox{$\underline{{#1}}$}}

\newcommand{\ivr}[1]{\mbox{$\tluste{#1}$}} 

\newcommand{\inum}[1]{\mbox{$\tluste{#1}$}}

\newcommand{\R}[0]{{\mathbb{R}}}

\newcommand{\IR}[0]{{\mathbb{IR}}}

\newcommand{\Ss}[0]{\mbox{\large$\Sigma$}}

\newcommand{\mmid}[0]{;\,}		

\newcommand{\seznam}[1]{{\{1, \ldots, {#1}\}}}

\newcommand{\st}[0]{{\ \ \mbox{subject to}\ \ }}
\DeclareMathOperator{\sgn}{sgn}	
\DeclareMathOperator{\diag}{diag}	

\def\nref#1{$(\ref{#1})$}

\newtheorem{proposition}{Proposition}

\theoremstyle{definition}

\newtheorem{example}{Example}

\begin{document}

\title{Robust optimal solutions in interval linear programming with forall-exists quantifiers}

\author{
  Milan Hlad\'{i}k\footnote{
Charles University, Faculty  of  Mathematics  and  Physics,
Department of Applied Mathematics, 
Malostransk\'e n\'am.~25, 11800, Prague, Czech Republic, 
e-mail: \texttt{milan.hladik@matfyz.cz}
}
}

\date{\today}
\maketitle

\begin{abstract}
We introduce a novel kind of robustness in linear programming. A solution $x^*$ is called robust optimal if for all realizations of objective functions coefficients and constraint matrix entries from given interval domains there are appropriate choices of the right-hand side entries from their interval domains such that  $x^*$ remains optimal. we propose a method to check for robustness of a given point, and also recommend how a suitable candidate can be found. We also discuss topological properties of the robust optimal solution set. We illustrate applicability of our concept in a transportation problem.
\end{abstract}


\section{Introduction}

Robustness in mathematical programming was intensively studied from diverse points of view \cite{BenBoy2006,BenNem2009,BenGor2004,BenNem2002, SoyMur2013}. Generally, robustness corresponds to stability of some key characteristics under limited input data change.
In case of uncertainties in the objective function only,  an optimal solution is usually called robust if the worst-case regret in the objective value is minimal.

One class of robustness is dealt with in the area of interval linear programming. Therein, we model uncertain parameters by intervals of admissible values and suppose that parameters can attain any value from their interval domains independently of other parameters. The effect of variations on the optimal value and interval solutions are the fundamental problems investigated 
\cite{AllNeh2013,Hla2012a,Hla2014:a}.
Concerning to robustness, \cite{Hla2014a} was devoted to stability of an optimal basis in  interval linear programming. In \cite{AveLeb2005,GabMur2010b,InuSak1995,MauLag1998}, the authors utilized maximum regret approach for finding robust solutions.
In multiobjective case, \cite{HlaSit2013,RivYag2013} studied robustness of a Parto optimal solution, and some  
specific nonlinear programming problems \cite{Hla2010c} were addressed in the context of interval robustness as well.

Recently, \cite{LiLuo2013,LuoLi2013a,LuoLi2014a} introduced a novel kind of interval robustness. They divided interval parameters into two sets, quantified respectively by universal and existential quantifiers.
Roughly speaking, an optimal solution is robust in this sense if for each realization of universally quantified interval parameter there is some realization of the existentially quantified parameters such that the solutions remains optimal. Such forall-exists quantified problems are also studied in the context of interval linear equations
\cite{Pop2012,PopHla2013,Sha2002}; imposing suitable quantifiers give us a more powerful technique in real-life problem modelling, and can more appropriately reflect various decision maker strategies.

This paper is a contribution to interval linear programming with quantified parameters. The robust optimal solutions considered must remain optimal for any admissible perturbation in the objective and matrix coefficients, compensated by suitable right-hand side change. We propose a method to check for this kind of robustness and present a cheap sufficient condition. We discuss properties of the set of all robust solutions, and propose a heuristic to find a robust solution.
We apply the robustness concept to transportation problem in a small numerical study. The equality form of linear programming is then extended to a general form with mixed equations and inequalities (Section~\ref{sGen}).

\paragraph*{Notation.} 
The $k$th row of a matrix $A$ is denoted as $A_{k*}$, and $\diag(s)$ stands for the diagonal matrix with entries given by $s$. The sign of a real $r$ is defined as $\sgn(r)=1$ if $r\geq0$ and $\sgn(r)=-1$ otherwise; for vectors the sign is meant entrywise.

An interval matrix is defined as
$$
\imace{A}:=\{A\in\R^{m\times n}\mmid \umace{A}\leq A\leq \omace{A}\},
$$
where $ \umace{A}$ and  $\omace{A}$, $\umace{A}\leq\omace{A}$, are given matrices. The midpoint and radius matrices are defined as
$$
\Mid{A}:=\frac{1}{2}(\umace{A}+\omace{A}),\quad
\Mid{A}:=\frac{1}{2}(\omace{A}-\umace{A}).
$$
Naturally, intervals and interval vectors are consider as special cases of interval matrices.
For interval arithmetic, we refer the readers to \cite{MooKea2009,Neu1990}, for instance.

\paragraph*{Problem formulation.} 
Consider a linear programming problem in the equality form
\begin{align}\label{lp}
\min c^Tx\st Ax=b,\ x\geq0.
\end{align}
Let $\imace{A}\in\IR^{m\times n}$, $\ivr{b}\in\IR^{m}$ and $\ivr{c}\in\IR^{n}$ be given. Let $x^*\in\R^n$ be a candidate robustly optimal solution. The problem states as follows:
\begin{quote}
For every $c\in\ivr{c}$ and $A\in\imace{A}$, does there exists $b\in\ivr{b}$ such that $x^*$ is optimal to \nref{lp}?
\end{quote}
In other words, we ask whether $x^*$ is robustly optimal in the sense that any change in $c$ and $A$ within the prescribed bounds can be compensated by an adequate change in $b$.
Thus, $\imace{A}$ and $\ivr{c}$ play role of uncertain parameters all realizations of which must be taken into account.
On the other hand, intervals in $\ivr{b}$ represent some reserves that we can utilize if necessary.

In \cite{Hla2012b}, it was shown that checking whether $x^*$ is optimal for all evaluations $c\in\ivr{c}$, with fixed $A$ and $b$, is a co-NP-complete problem. Since the class of problems studied in this manuscript covers this as a sub-class, we have as a consequence that our problem is co-NP-complete problem as well. This practically means that we hardly can hope for a polynomial time verification of robust optimality.

\section{Checking robust optimality}\label{sEqForm}

Let $I:=\{i=1,\dots,n\mmid x^*_i=0\}$ be the set of active indices of $x^*$. It is well known that $x^*$ is optimal if and only if $x^*$ is feasible, and there is no strictly better solution in the tangent cone at $x^*$ to the feasible set. In other words, the linear system
\begin{align}\label{optCond}
c^Tx=-1,\ \ Ax=0,\ \ x_i\geq0,\ i\in I,
\end{align}
has no solution. We refer to this conditions as \emph{feasibility} and \emph{optimality}. In order that $x^*$ is robustly optimal, both conditions must hold with the given forall-exists quantifiers. Notice that only the entries of $A$ are situated in both conditions. Since there is the universal quantifier associated with $A$, we can check for feasibility and optimality separately.

\paragraph*{Feasibility.}
We have to check that for any $A\in\imace{A}$ there is $b\in\ivr{b}$ such that $Ax^*=b$. This is well studied problem and $x^*$ satisfying this property is called tolerance (or tolerable) solution; see \cite{Fie2006,Pop2013a,Sha2002,Sha2004}. By \cite[Thm. 2.28]{Fie2006}, $x^*$ is a tolerance solution if and only if it satisfies
\begin{align}\label{optFeas}
|\Mid{A}x^*-\Mid{b}|+\Rad{A}|x^*|\leq\Rad{b}.
\end{align}
Thus, the feasibility question is easily answered.

\paragraph*{Optimality.} 
Denote by $A_I$ the restriction of $A$ to the columns indexed by $I$, and denote by $A_J$ the restriction to the columns indexed by  $J:=\seznam{n}\setminus I$. In a similar manner we use $I$ and $J$ as sub-indices for other matrices and vectors.

We want to check whether \nref{optCond} is infeasible for any  $A\in\imace{A}$ and $c\in\ivr{c}$. By \cite{Hla2013b}, this is equivalent to infeasibility of the system
\begin{align*}
(\uvr{c_I})^Tx_I+(\Mid{c_J})^Tx_J&\leq(\Rad{c_J})^T|x_J|-1,\\
-(\ovr{c_I})^Tx_I-(\Mid{c_J})^Tx_J&\leq(\Rad{c_J})^T|x_J|+1,\\
\uvr{A_I}\,x_I+\Mid{A_J}x_J&\leq\Rad{A_J}|x_J|,\\
-\ovr{A_I}\,x_I-\Mid{A_J}x_J&\leq\Rad{A_J}|x_J|,\\
x_I&\geq0.
\end{align*}
Due to the absolute values, the system is nonlinear in general, and it is the reason why checking robust optimality is co-NP-hard. Equivalently, this system is infeasible if and only if
\begin{subequations}\label{optAeExp}
\begin{align}
(\uvr{c_I})^Tx_I+(\Mid{c_J}-\Rad{c_J}\diag(s))^Tx_J&\leq-1,\\
-(\ovr{c_I})^Tx_I-(\Mid{c_J}+\Rad{c_J}\diag(s))^Tx_J&\leq1,\\
\uvr{A_I}\,x_I+(\Mid{A_J}-\Rad{A_J}\diag(s))x_J&\leq0,\\
-\ovr{A_I}\,x_I+(\Mid{A_J}+\Rad{A_J}\diag(s))x_J&\leq0,\\
x_I&\geq0
\end{align}
\end{subequations}
is infeasible for any sign vector $s\in\{\pm1\}^{|J|}$, where $|J|$ denotes the cardinality of $J$.

The system \nref{optAeExp} is linear, however, we have to verify infeasibility $2^{|J|}$ of instances. When $x^*$ is a basic feasible solution, then $|J|\leq m\leq n$. Thus, the number usually grows exponentially with respect to $m$, but not necessarily with respect to $n$. Therefore, we possibly can solve large problems provided the number of equations is low.

\subsection{Sufficient condition}\label{ssSufCond}

Since the number of systems \nref{optAeExp} can be very large, an easily computable sufficient condition for robust optimality is of interest.

Let us rewrite \nref{optCond} as
\begin{align}\label{optCond2}
c_I^Tx_I+c_J^Tx_J=-1,\ \ A_Ix_I+A_Jx_J=0,\ \ x_I\geq0.
\end{align}
According to the Farkas lemma \cite{Fie2006,Schr1998}, this system is infeasible if and only if the dual system
\begin{align}\label{optCondDual}
A_I^Tu \leq c_I,\ \ A_J^Tu = c_J
\end{align}
is feasible. Thus, in order that the optimality condition holds true, the linear system \nref{optCondDual} must be feasible for each $A\in\imace{A}$ and $c\in\ivr{c}$. 

If $x^*$ is a basic non-degenerate solution, then $A_J$ is square. If it is nonsingular in addition, then the system $A_J^Tu = c_J$ has a unique solution, and it suffices to check if the solution satisfies the remaining inequalities. Extending this idea to the interval case, consider the solution set defined as 
$$
\{u\in\R^m\mmid  \exists A_J\in\imace{A}_J\exists c_J\in\ivr{c}_J:
A_J^Tu = c_J\}.
$$
There are plenty of methods to find an interval enclosure (superset) $\ivr{u}$ of this solution set; see e.g.\ \cite{Fie2006,Hla2014b,MooKea2009,Neu1990}.
Now, if 
$$
\ovr{\imace{A}_I^T\ivr{u}} \leq \unum{c}_I,
$$
where the left-hand side is evaluated by interval arithmetic, then we are sure that \nref{optCondDual} has a solution in $\ivr{u}$ for each realization of interval data, and therefore the optimality criterion is satisfied.

If $x^*$ is a basic degenerate solution, we can adopt a sufficient condition for checking similar kind of robust feasibility of mixed system of equations and inequalities proposed recently in \cite{Hla2013b}. We will briefly recall the method.
First, solve the linear program
\begin{align*}
\max\alpha\st
(\Mid{A_I})^Tu+\alpha e \leq \Mid{c_I},\ \ (\Mid{A_J})^Tu = \Mid{c_J},
\end{align*}
where $e$ is the all-one vector. Let $u^*$ be its optimal solution. Let $B$ be an orthogonal basis of the null space of $(\Mid{A_J})^T$ and put $d:=Bu^*$. 
Now, compute an enclosure $\ivr{u}\in\IR^m$ of the solutions set 
$$
\{u\in\R^m\mmid  \exists A_J\in\imace{A}_J\exists c_J\in\ivr{c}_J:
A_J^Tu = c_J,\  Bu=d\}.
$$
Finally, if 
$$
\ovr{\imace{A}_I^T\ivr{u}} \leq \unum{c}_I,
$$
the the optimality criterion is satisfied.

\subsection{Seeking for a candidate}\label{ssCand}

If we are not given a candidate vector $x^*$ for a robust optimal solution, then it may be a computationally difficult problem to find a robust optimal solution or to prove that there is no one. Below, we propose a simple heuristic for finding a promising candidate.

A candidate should be robustly feasible. 
The condition \nref{optFeas} is can be rewritten in a linear form as
\begin{align*}
(\Mid{A}x^*-\Mid{b})+\Rad{A}x^*\leq\Rad{b},\ \ 
-(\Mid{A}x^*-\Mid{b})+\Rad{A}x^*\leq\Rad{b},
\end{align*}
or, equivalently, as
\begin{align}\label{ineqSetRob}
\omace{A}x^*\leq\ovr{b},\ \ 
\umace{A}x^*\geq\uvr{b}.
\end{align}
This motivates us to find a good candidate $x^*$ as an optimal solution of the linear program
\begin{align*}
\min (\Mid{c})^Tx\st x\in\mna{F},
\end{align*}
where
\begin{align*}
\mna{F}:=\{x\in\R^n\mmid
\omace{A}x\leq\ovr{b},\ \ 
\umace{A}x\geq\uvr{b},\ \ 
x\geq0\}.
\end{align*}

\subsection{The set of robust solutions in more detail}

Let us denote by $\Ss$ the set of all robust optimal solutions.

\begin{proposition}\label{propSsEqUniConv}
$\Ss$ is formed by a union of at most $\binom{n}{\lfloor n/2\rfloor}$ convex polyhedral sets.
\end{proposition}

\begin{proof}
Each $x\in\Ss$ must lie in $\mna{F}$ and must satisfy the optimality criterion. Since the optimality criterion does not depend directly on $x$, but only on the active set $I$ of $x$, we have that
\begin{align*}
\mna{F}_I:=\mna{F}\cap
 \{x\in\R^n\mmid x_i=0,\ i\in I,\ x_i>0,\ i\not\in I\}
\end{align*}
either whole lies in $\Ss$, or is disjoint with $\Ss$. Hence $\Ss$ is formed by a union of the sets $\mna{F}_I$ for several index sets $I\subseteq \seznam{n}$. Since 
\begin{align*}
\mna{F}_I\subseteq\Ss\wedge I\subseteq I'\ 
\Rightarrow\ \mna{F}_{I'}\subseteq\Ss,
\end{align*}
we can replace the sets $\mna{F}_I$ by
\begin{align*}
\tilde{\mna{F}}_I:=\mna{F}\cap
 \{x\in\R^n\mmid x_i=0,\ i\in I\}.
\end{align*}
Now, since $\tilde{\mna{F}}_I\supseteq\tilde{\mna{F}}_{I'}$ for $I\subseteq I'$, not all subsets of $\seznam{n}$ have to be taken into account. By Sperner's theorem (see, e.g., \cite{MatNes2008}), only  $\binom{n}{\lfloor n/2\rfloor}$ of them it is sufficient to consider.
\end{proof}

As illustrated by the following example, the robust solution set $\Ss$ needn't be topologically connected.

\begin{example}
Consider the problem
\begin{align*}
\min x_1+x_2+c_3x_3
\st x_1+x_2+x_3=1,\ x_1-x_2=b_2,\ x_1,x_2,x_3\geq0,
\end{align*}
where $c_3\in[0.5,1.5]$ and $b_2\in[-1,1]$. The robust feasible set $\mna{F}$ is formed by a triangle with vertices $(1,0,0)$, $(0,1,0)$ and $(0,0,1)$.
Concerning optimality, the system \nref{optCond} reads
\begin{align}\label{sysRobOptEx}
x_1+x_2+c_3x_3=-1,\ \ 
x_1+x_2+x_3=0,\ \ 
x_1-x_2=0,\ \ 
x_I\geq0.
\end{align}
If $3\not\in I$, then \nref{sysRobOptEx} has a solution $x=(1,1,-2)$ when $c_3=1.5$. Thus, it must be $3\in I$. If $I=\{3\}$, then \nref{sysRobOptEx} has a solution $x=(-1,-1,2)$ when $c_3=0.5$. If $I=\{1,3\}$, then \nref{sysRobOptEx} has no solution for any $c_3$, and the corresponding $\tilde{\mna{F}}_I=\{(0,1,0)\}$. Similarly, for $I=\{2,3\}$, the system \nref{sysRobOptEx} has no solution for any $c_3$, and $\tilde{\mna{F}}_I=\{(1,0,0)\}$. 
In summary, the robust solution set $\Ss$ consists of two isolated points $(1,0,0)$ and $(0,1,0)$.
\end{example}

\section{Applications}

\subsection{Transportation problem}

Since linear programming is so widely used technique, the proposed concept of robust solution and the corresponding methodology is applicable in many practical problems. These problems include transportation problem and flows in networks, among others, in which the constraint matrix $A$ represents an incidence matrix of a (undirected or directed) graph. By imposing suitable intervals ($[0,1]$ or $[-1,1]$) as the matrix entries, we can model uncertainty in the knowledge of the edge existence.

More concretely, consider a transportation problem
\begin{align*}
\min\ &\sum_{i=1}^m\sum_{j=1}^nc_{ij}x_{ij}\\
\st 
&\sum_{i=1}^m\alpha_{ij}x_{ij}=b_j,\quad j=1,\dots,n,\\
&\sum_{j=1}^n\alpha_{ij}x_{ij}=a_i,\quad i=1,\dots,m,\\
&x_{ij}\geq0,\quad  i=1,\dots,m,\ j=1,\dots,n,
\end{align*}
where $c_{ij}\in\inum{c}_{ij}$, $a_i\in\inum{a}_i$ and $b_j\in\inum{b}_j$. In contrast to the standard formulation $\alpha_{ij}\in\{0,1\}$ and in order to obtain interval parameters, we allow $\alpha_{ij}$ to attain values in the interval $[0,1]$.

Robustness here means that an optimal solution $x^*$ remains optimal for any $c_{ij}\in\inum{c}_{ij}$. Moreover, $x^*$ should also remain optimal even when some selected edges are removed. The edge removal could be compensated by a suitable change of $a_i\in\inum{a}_i$ and $b_j\in\inum{b}_j$. Herein,the intervals  $\inum{a}_i$ and $\inum{b}_j$ are interpreted as tolerances in supplies and demands.

\begin{example}
For concreteness, let
\begin{align*}
C=\begin{pmatrix}20& 30& 10\\10& 20& 50\\40& 10& 20\end{pmatrix},\quad
a=(100, 160, 250),\quad
b=(150, 210, 150).
\end{align*}
Suppose that the  objective coefficients $c_{ij}$ are known with $10\%$ precision only. Next suppose that the supplies and demands have $10\%$ tolerance in which they can be adjusted. Eventually, suppose that the connections from the second supplier to the second and third demanders, and from the third supplier to the first demander are all uncertain. Thus, we have interval data
\begin{align*}
\imace{C}&=\begin{pmatrix}
[18,22]& [27,33]& [9,11]\\
[9,11]& [18,22]& [45,55]\\
[36,44]& [9,11]& [18,22]
\end{pmatrix},\\
\ivr{a}&=([90,110],\, [144,176],\, [225,275]),\\
\ivr{b}&=([135,165],\,[189,231],\,[135,165])\\
\inum{\alpha}_{22}&=\inum{\alpha}_{23}=\inum{\alpha}_{31}=[0,1],\ \ 
\inum{\alpha}_{ij}=1,\ (i,j)\not\in\{(2,2),(2,3),(3,1)\}.
\end{align*}

For the midpoint data, the optimal solution is 
\begin{align*}
\begin{pmatrix}0& 0& 100\\150& 10& 0\\0& 200& 50\end{pmatrix}.
\end{align*}
It is robustly feasible, however, it is not robustly optimal.

Let us try our method from Section~\ref{ssCand}. It recommends to solve the problem
\begin{align*}
\min\ &\sum_{i=1}^m\sum_{j=1}^n\Mid{c}_{ij}x_{ij}\\
\st 
&\sum_{i=1}^m\onum{\alpha}_{ij}x_{ij}\leq \onum{b}_j,\ \ 
\sum_{i=1}^m\unum{\alpha}_{ij}x_{ij}\geq \unum{b}_j,
 \quad j=1,\dots,n,\\
&\sum_{j=1}^n\onum{\alpha}_{ij}x_{ij}\leq\onum{a}_i,\ \ 
\sum_{j=1}^n\unum{\alpha}_{ij}x_{ij}\geq\unum{a}_i,
 \quad i=1,\dots,m,\\
&x_{ij}\geq0,\quad  i=1,\dots,m,\ j=1,\dots,n.
\end{align*}
Its solution is
\begin{align*}
\begin{pmatrix}0& 0& 99\\144& 0& 0\\0& 189& 36\end{pmatrix}.
\end{align*}
It turns out that is is both robustly feasible and optimal, so it can serve as a robust solution in question. As the sufficient condition fails, optimality must have been verified by the exhausting feasibility checking of 16 systems of type \nref{optAeExp}. Nevertheless, if the edge $(2,2)$ becomes certain, and only the others are uncertain, then the sufficient condition succeeds.
\end{example}

\begin{example}\label{exDpNum}
We carried out a limited numerical study about what is the efficiency of the sufficient condition and the heuristic for finding a candidate. In the transportation problem with given dimensions $m$ and $n$, we randomly chosen $C$, $a$ and $b$. In $C$, there were $10\%$ of randomly chosen entries subject to $10\%$ relative uncertainty. Tolerances for supplies and demands were also $10\%$. A given number randomly selected edges were considered as uncertain, i.e., the coefficients by $x_{ij}$ ranged in $[0,1]$.

\begin{table}[t]
\caption{(Example~\ref{exDpNum}) Computing time and percentual rate of finding robust optimal solution for different dimensions and numbers of uncertain edges in the transportation problem.\label{tabDpNum}}
\begin{center}
\tabcolsep=5pt
\renewcommand\arraystretch{1.1}
\begin{tabular}{cccccc}
 \toprule
$m$ & $n$ &  $[0,1]$-edges & candidate time (in $s$)
 & robustness time (in $s$) & success rate (in \%)\\
\midrule 
\multirow{3}{*}{5} & \multirow{3}{*}{10}
   & 2 & 0.03138 & 0.01268 & 25.66 \\
  && 4 & 0.03155 & 0.00673 & 10.63 \\
  && 6 & 0.03178 & 0.00415 & 4.89 \\
 \cline{3-6}
\multirow{3}{*}{10} & \multirow{3}{*}{15}
    & 3 & 0.06980 & 0.02139 & 18.84 \\
   && 5 & 0.06904 & 0.01447 & 11.90 \\
   && 7 & 0.06873 & 0.01011 &  7.46 \\
 \cline{3-6}
\multirow{3}{*}{10} & \multirow{3}{*}{30}
    & 4 & 0.1364 & 0.02919 & 9.30 \\
   && 6 & 0.1370 & 0.02035 & 5.96 \\
   && 8 & 0.1336 & 0.01457 & 3.94 \\
 \cline{3-6}
\multirow{3}{*}{20} & \multirow{3}{*}{50}
    &  4 & 0.4737 & 0.11660 & 4.47 \\
   &&  7 & 0.4520 & 0.08784 & 2.45 \\
   && 10 & 0.4265 & 0.06585 & 1.68 \\
 \cline{3-6}
\multirow{3}{*}{30} & \multirow{3}{*}{70}
    &  6 & 1.0825 & 0.2666 & 1.04 \\
   &&  9 & 1.0119 & 0.2177 & 0.80 \\
   && 12 & 0.9657 & 0.1853 & 0.42 \\
 \cline{3-6}
\multirow{3}{*}{50} & \multirow{3}{*}{100}
    & 1 & 1.6982 & 0.04426 & 0 \\
   && 3 & 1.8670 & 0.04975 & 0 \\
   && 5 & 2.0025 & 0.08210 & 0 \\
\bottomrule
\end{tabular}
\end{center}
\end{table}

Table~\ref{tabDpNum} displays the results. Each row is a mean of 10000 runs, and shows the average running time in seconds and the success rate. The success rate measures for how many instances the heuristic found a candidate that was after that verified as a robust optimal solution by the sufficient condition. This means that the number of robust solution can be higher, but we were not able to check it because of its intractability. The displayed running time concerns both the heuristic for finding a suitable candidate and the sufficient condition for checking robustness.

The results show that in low computing time we found robust optimal solutions in 5\% to 15\% of the small dimensional cases. In large dimensions, the number of robust solutions is likely to be small. Even when we decreased the number of uncertain edges, the sufficient condition mostly failed. This may be due to 500 interval costs in the last data set.
\end{example}

\subsection{Nutrition problem}

The diet problem is the classical linear programming problem, in which a combination of $n$ different types of food must be found such that
$m$ nutritional demands are satisfied and the overall cost in minimal.
The mathematical formulation has exactly the form of \nref{lp}, where $x_j$ be the number of units of food $j$ to be consumed, $b_i$ is the required amount of nutrient $j$, $c_j$ is the price per unit of food $j$, and $a_{ij}$ is the the amount of nutrient $j$ contained in one unit of food $i$.

Since the amounts of nutricients is foods are not constant, it is reasonable to consider intervals of possible ranges instead of fixed values. The same considerations apply for the costs.
The requirements on nutritional demands to be satisfied as equations are too strict. Usually, there are large tolerances on the amount of consumed nutricients (such as calories, proteins, vitamins, etc.), which leads to quite wide intervals of admissible tolerances for the entries of $b$. In this interval valued diet problem, we would like to find optimal solution $x$ that is robustly feasible in the sense that for each possible instance of $A$ there is an admissible $b$ such that $Ax=b$. This model is exactly the robustness model we are dealing with in this paper.

\begin{example}
Consider Stigler's nutrition model \cite{Dan1963}, the GAMS model file containing the data is posted at \url{http://www.gams.com/modlib/libhtml/diet.htm}.
The problem consists of $m=9$ nutrients and $n=20$ types of food. The data in $A$ are already normalized such that it gives nutritive values of foods per dollar spent. This means that the objective is $c=(1,\dots,1)^T$.

Suppose that the entries of $A$ can vary up to $5\%$ of their nominal values, and the tolerances in $b$ are $10\%$. Then the method from Section~\ref{ssCand} finds the solution
\begin{align*}
x^*
=(&0.0256, 0.0067, 0.0429, 0, 0, 0.0015, 0.0245,\\
  &0.0108, 0, 0, 0, 0.0109, 0, 0, 0.0016, 0, 0, 0, 0, 0)^T.
\end{align*}
Even though our sufficient condition fails, it turns out by checking infeasibility of \nref{optAeExp} for each sign vector that this solution is robustly optimal. 
\end{example}

\section{General form of interval linear programming}\label{sGen}

For the sake of simplicity of exposition, we considered the equality form of linear programming \nref{lp} in the first part of this paper. It is well known in interval linear programming that different forms are not equivalent to each other \cite{Hla2012a} since transformations between the formulations lead to dependencies between interval coefficients. That is why we will consider a general form of interval linear programming in this section and extend the results developed so far.

The general form with $m$ equation, $m'$ inequalities and variables $x\in\R^n$, $y\in\R^{n'}$ reads
\begin{align}\label{lpGen}
\min c^Tx+d^Ty\st Ax+By=b,\ Cx+Dy\leq a,\ x\geq0,
\end{align}
where  $a\in\ivr{a}$, $b\in\ivr{b}$, $c\in\ivr{c}$, $d\in\ivr{d}$, $A\in\imace{A}$,  $B\in\imace{B}$,  $C\in\imace{C}$ and  $D\in\imace{D}$. Let $(x^*,y^*)$ be a candidate solution. The problem now states as follows. 
\begin{quote}
For every $c\in\ivr{c}$, $d\in\ivr{d}$, $A\in\imace{A}$, $B\in\imace{B}$,  $C\in\imace{C}$,  $D\in\imace{D}$, does there exist $a\in\ivr{a}$ and $b\in\ivr{b}$ such that $(x^*,y^*)$ is optimal to \nref{lpGen}?
\end{quote}
As before, we will study feasibility and optimality conditions separately.

\subsection{Feasibility}

Here, we have to check whether for each  $A\in\imace{A}$, $B\in\imace{B}$,  $C\in\imace{C}$ and  $D\in\imace{D}$, there are $a\in\ivr{a}$ and $b\in\ivr{b}$ such that $Ax^*+By^*=b$ and $Cx^*+Dy^*\leq a$. We can check equations and inequalities independently. Equations are dealt with in a similar manner as in Section~\ref{sEqForm}, and the sufficient and necessary condition is
\begin{align}\label{condGenFeasEq}
|\Mid{A}x^*+\Mid{B}y^*-\Mid{b}|+\Rad{A}x^*+\Rad{B}|y^*|\leq\Rad{b}.
\end{align}
For inequalities, we have the following characterisation.

\begin{proposition}
For each $C\in\imace{C}$ and  $D\in\imace{D}$, there is $a\in\ivr{a}$ such that $Cx^*+Dy^*\leq a$ if and only if
\begin{align}\label{condGenFeasIneq}
\omace{C}x^*+\Mid{D}y^*+\Rad{D}|y^*|\leq\ovr{a}.
\end{align}
\end{proposition}

\begin{proof}
For each $C\in\imace{C}$ and  $D\in\imace{D}$ we have
\begin{align*}
Cx^*+Dy^*
&=Cx^*+\Mid{D}y^*+(D-\Mid{D})y^*
\leq Cx^*+\Mid{D}y^*+|D-\Mid{D}||y^*|\\
&\leq \omace{C}x^*+\Mid{D}y^*+\Rad{D}|y^*|.
\end{align*}
This inequality chain holds as equation for $C:=\omace{C}$ and $D:=\Mid{D}+\Rad{D}\diag(\sgn(y^*))$ since $|y^*|=\diag(\sgn(y^*))y^*$. That is, the largest value of the left-hand side is attained for this setting. Therefore, feasibility condition holds true if and only if the inequality is satisfied for this setting of $C$ and $D$, and for the largest possible right-hand side vector $a:=\ovr{a}$.
\end{proof}

\subsection{Optimality} 

For checking optimality we have to define the active set first. For nonnegativity constraints, we can use the standard definition $I:=\{i=1,\dots,n\mmid x^*_i=0\}$. However, we face a problem to define the active set for the other inequalities due to the variations in $C$ and $D$. Fortunately, we can define it as follows.

\begin{proposition}\label{propGenOpt}
Each instance of the inequality system $Cx^*+Dy^*\leq a$, $C\in\imace{C}$, $D\in\imace{D}$, with a suitable $a\in\ivr{a}$ includes the index set
\begin{align*}
K:=\{k\mmid 
\umace{C}_{k*}x^*+\Mid{D}_{k*}y^*-\Rad{D}_{k*}|y^*|\geq \unum{a}_k\}
\end{align*}
as a subset of its active set. 
Moreover, $K$ is attained as an active set for $C:=\umace{C}$, $D:=\Mid{D}-\Rad{D}\diag(\sgn(y^*))$ and $a:=\max\{\uvr{a},Cx^*+Dy^*\}$.
\end{proposition}

\begin{proof}
First, we show that $K$ is attained for $C:=\umace{C}$, $D:=\Mid{D}-\Rad{D}\diag(\sgn(y^*))$, and $a:=\max\{\uvr{a},Cx^*+Dy^*\}\in\ivr{a}$.
The condition $Cx^*+Dy^*\leq\ovr{a}$, and hence also $a\in\ivr{a}$, follows from feasibility of $(x^*,y^*)$, so $a$ is well defined.
For $k\in K$, we have
\begin{align*}
{C}_{k*}x^*+{D}_{k*}y^*
=\umace{C}_{k*}x^*+\Mid{D_{k*}}y^*-\Rad{D_{k*}}|y^*|
\geq \unum{a}_k,
\end{align*}
whence ${C}_{k*}x^*+{D}_{k*}y^*=a_k$. For $k\not\in K$, we have
\begin{align*}
{C}_{k*}x^*+{D}_{k*}y^*
=\umace{C}_{k*}x^*+\Mid{D_{k*}}y^*-\Rad{D_{k*}}|y^*|
< \unum{a}_k=a_k.
\end{align*}

Now, let $C\in\imace{C}$, $D\in\imace{D}$ and $k\in K$ be arbitrary. From the feasibility of $(x^*,y^*)$ and
\begin{align*}
\unum{a}_k
\leq\umace{C}_{k*}x^*+\Mid{D_{k*}}y^*-\Rad{D_{k*}}|y^*|
\leq{C}_{k*}x^*+{D}_{k*}y^*
\end{align*}
we can put $a_k:={C}_{k*}x^*+{D}_{k*}y^*\in\inum{a}_k$. Therefore, $k$ lies in the active set corresponding to $C$, $D$ and $a$.
\end{proof}

Notice that the larger the active set the better since we have more constraints in the optimality criterion and the solution is more likely optimal.
Proposition~\ref{propGenOpt} says that we can take $K$ as the active set to the interval inequalities. Since for each $C\in\imace{C}$ and $D\in\imace{D}$, this $K$ is the smallest active set, it is the worst case scenario that we can imagine. Similarly, the right-hand side vector $a$ from Proposition~\ref{propGenOpt} is the best response: If we decrease it, then $(x^*,y^*)$ or $a$ becomes infeasible, and if we increase it, then the active set becomes smaller. 

To state the optimality criterion comprehensively, we have to introduce some notation first. Let $\tilde{A}:=A_I$, $\tilde{B}:=(A_J\mid B)$, $\tilde{c}:=c_I$, $\tilde{d}:=(c_J, d)$, $\tilde{x}:=x_I$, $\tilde{y}:=(x_J,y)$. Let $\tilde{C}$ be the restriction of $C_I$ to the rows indexed by $K$, and similarly $\tilde{D}$ be a restriction of $(C_J\mid D)$ to the rows indexed by $K$.

For a concrete setting $a\in\ivr{a}$, $b\in\ivr{b}$, $c\in\ivr{c}$, $d\in\ivr{d}$, $A\in\imace{A}$,  $B\in\imace{B}$,  $C\in\imace{C}$ and  $D\in\imace{D}$, a feasible solution $(x^*,y^*)$ is optimal if and only if
\begin{align}\label{sysGenOpt}
\tilde{c}^T\tilde{x}+\tilde{d}^T\tilde{y}\leq-1,\ \ 
\tilde{A}\tilde{x}+\tilde{B}\tilde{y}=0,\ \ 
\tilde{C}\tilde{x}+\tilde{D}\tilde{y}\leq0,\ \ 
\tilde{x}\geq0
\end{align}
has no solution. In order that $(x^*,y^*)$ is robustly optimal, this systems has to be infeasible for each realization from the given intervals. By \cite{Hla2013b}, \nref{sysGenOpt} is infeasible for each realization if and only if the system
\begin{align*}
\uvr{\tilde{c}}^T\tilde{x}+(\Mid{\tilde{d}})^T\tilde{y}
 &\leq(\Rad{\tilde{d}})^T|\tilde{y}|-1,\\ 
\umace{\tilde{A}}\tilde{x}+\Mid{\tilde{B}}\tilde{y}
 &\leq\Rad{\tilde{B}}|\tilde{y}|,\\ 
-\omace{\tilde{A}}\tilde{x}-\Mid{\tilde{B}}\tilde{y}
 &\leq\Rad{\tilde{B}}|\tilde{y}|,\\ 
\umace{\tilde{C}}\tilde{x}+\Mid{\tilde{D}}\tilde{y}
 &\leq\Rad{\tilde{D}}|\tilde{y}|,\\
\tilde{x}&\geq0
\end{align*}
is infeasible. Even though we reduced infeasibility checking from infinitely many systems to only one, the resulting system in nonlinear. As in Section~\ref{sEqForm}, we can formulate it equivalently as infeasibility of
\begin{align*}
\uvr{\tilde{c}}^T\tilde{x}
 +(\Mid{\tilde{d}}-\Rad{\tilde{d}}\diag(s))^T\tilde{y}
 &\leq-1,\\ 
\umace{\tilde{A}}\tilde{x}
 +(\Mid{\tilde{B}}-\Rad{\tilde{B}}\diag(s))\tilde{y}
 &\leq 0,\\ 
-\omace{\tilde{A}}\tilde{x}
 -(\Mid{\tilde{B}}+\Rad{\tilde{B}}\diag(s))\tilde{y}
 &\leq 0,\\ 
\umace{\tilde{C}}\tilde{x}
 +(\Mid{\tilde{D}}-\Rad{\tilde{D}}\diag(s))\tilde{y}
 &\leq 0,\\
\tilde{x}&\geq0
\end{align*}
for every  $s\in\{\pm1\}^{n'+|J|}$. Now, we have to check infeasibility of $2^{n'+|J|}$ linear systems, which is large but finite. In case there are few sign-unrestricted variables and few positive components in $x^*$, the number fo systems can be acceptable for computation.

\subsection{Sufficient condition} 

Similarly as in Section~\ref{ssSufCond}, we can derive a sufficient condition for optimality checking. We discuss it briefly and refer to \cite{Hla2013b} for more details. By Farkas lemma, the optimality criterion holds true if and only if the dual system
\begin{align}\label{optGenCondDual}
\tilde{A}^Tu-\tilde{C}^Tv\leq \tilde{c},\ \ 
\tilde{B}^Tu-\tilde{D}^Tv  =  \tilde{d},\ \ 
v \geq 0
\end{align}
is feasible for each interval setting.
First, solve the linear program
\begin{align*}
\max\alpha\st
(\Mid{\tilde{A}})^Tu-(\Mid{\tilde{C}})^Tv+\alpha e
    \leq\Mid{\tilde{c}},\ \ 
(\Mid{\tilde{B}})^Tu-(\Mid{\tilde{D}})^Tv = \Mid{\tilde{d}},\ \ 
v \geq \alpha e.
\end{align*}
Let $(u^*,v^*)$ be its optimal solution, and let $(\hat{B}\mid-\hat{D})$ be an orthogonal basis of the null space of $((\Mid{\tilde{B}})^T\mid-(\Mid{\tilde{D}})^T)$ and put $\hat{d}:=\hat{B}u^*-\hat{D}v^*$. 
Compute an interval enclosure $(\ivr{u},\ivr{v})$ to the solution set of the square interval system
\begin{align*}
\{(u,v)\mmid \exists \tilde{B}\in\tilde{\imace{B}}
\exists \tilde{D}\in\tilde{\imace{D}}
\exists \tilde{d}\in \tilde{\imace{d}}:
\tilde{B}^Tu-\tilde{D}^Tv  =  \tilde{d},\ \ 
\hat{B}u-\hat{D}v=\hat{d}\},
\end{align*}
and check whether $\uvr{v}\geq0$ and
$$
\ovr{\tilde{\imace{A}}^T\ivr{u}-\tilde{\imace{C}}^T\ivr{v}}
\leq \tilde{\uvr{c}}.
$$
If they are satisfied, then \nref{optGenCondDual} has a solution in the set $(\ivr{u},\ivr{v})$ for each interval realization, and we can claim that optimality criterion holds true.

\subsection{Seeking for a candidate} 

Herein, we generalize the heuristic from Section~\ref{ssCand} to find a good candidate for robust optimal solution. Concerning the feasibility question, the conditions \nref{condGenFeasEq} and \nref{condGenFeasIneq} are not convenient due to their nonlinearities. Thus, we state an equivalent, linear form of feasibility testing.

\begin{proposition}
A vector $(x,y)$ is robustly feasible if and only if $y$ has the form of $y=y^1-y^2$ such that 
\begin{subequations}\label{condGenFeasProp}
\begin{align}\label{condGenFeasProp1}
\omace{A}x+\omace{B}y^1-\umace{B}y^2&\leq\ovr{b},\\ 
\label{condGenFeasProp2}
\umace{A}x+\umace{B}y^1-\omace{B}y^2&\geq\uvr{b},\\ 
\label{condGenFeasProp3}
\omace{C}x+\omace{D}y^1-\umace{D}y^2&\leq\ovr{a},\\ 
x,y^1,y^2&\geq0.
\end{align}
\end{subequations}
\end{proposition}

\begin{proof}
Let $(x,y^1,y^2)$ be a solution to \nref{condGenFeasProp}. For any $A\in\imace{A}$ and $B\in\imace{B}$ we have
\begin{align*}
Ax+B(y^1-y^2)\leq\omace{A}x+\omace{B}y^1-\umace{B}y^2&\leq\ovr{b},
\end{align*}
and
\begin{align*}
Ax+B(y^1-y^2)\geq\umace{A}x+\umace{B}y^1-\omace{B}y^2&\geq\uvr{b},
\end{align*}
whence $Ax+B(y^1-y^2)\in\ivr{b}$. For any  $C\in\imace{C}$ and  $D\in\imace{D}$ we have
\begin{align*}
Cx+D(y^1-y^2)\leq\omace{C}x+\omace{D}y^1-\umace{D}y^2&\leq\ovr{a},
\end{align*}
so $(x,y^1-y^2)$ is robustly feasible.

Conversely, let $(x,y)$ be robustly feasible. Put $y^1:=y^+$ and $y^2:=y^-$, the positive and negative parts of $y$. From \nref{condGenFeasEq}, we derive
\begin{align*}
|\Mid{A}x+\Mid{B}(y^1-y^2)-\Mid{b}|+\Rad{A}x+\Rad{B}(y^1+y^2)
\leq\Rad{b}.
\end{align*}
This inequality gives rise to two linear inequalities:
\begin{align*}
\Mid{A}x+\Mid{B}(y^1-y^2)-\Mid{b}+\Rad{A}x+\Rad{B}(y^1+y^2)
&\leq\Rad{b},\\
-\Mid{A}x-\Mid{B}(y^1-y^2)+\Mid{b}+\Rad{A}x+\Rad{B}(y^1+y^2)
&\leq\Rad{b},
\end{align*}
which are equivalent to \nref{condGenFeasProp1}--\nref{condGenFeasProp2}. Similarly, \nref{condGenFeasIneq} implies
\begin{align*}
\omace{C}x+\Mid{D}(y^1-y^2)+\Rad{D}(y^1+y^2)\leq\ovr{a},
\end{align*}
which is equivalent to \nref{condGenFeasProp3}.
\end{proof}

Now, we recommend to take as a candidate solution the pair $(x^*,y^{*1}-y^{*2})$, where $(x^*,y^{*1},y^{*2})$ is an optimal solution of the linear program
\begin{align*}
\min (\Mid{c})^Tx+(\Mid{d})^Ty^1-(\Mid{d})^Ty^2
\st (\ref{condGenFeasProp}).
\end{align*}

\subsection{The set of robust solutions in more detail}

As before, we denote by $\Ss$ the set of all robust optimal solutions and state to following topological results on it.

\begin{proposition}\label{propSsGenUniConv}
$\Ss$ is formed by a union of at most $\binom{m'}{\lfloor m'/2\rfloor}\binom{n}{\lfloor n/2\rfloor}$ convex polyhedral sets.
\end{proposition}

\begin{proof}
Each $x\in\Ss$ must satisfy both the feasibility and optimality criteria. The robust feasible set is a convex polyhedral set, so we focus on the optimality issue. The optimality depends only on the active sets $I$ and $K$, not on the concrete value of $x$. Given $I$ and $K$, the corresponding set
\begin{align*}
\mna{F}\cap
 \{(x,y)\in\R^{n+n'}\mmid &x_i=0,\ i\in I,\ x_i>0,\ i\not\in I,\\ 
&\umace{C}_{k*}x+\Mid{D}_{k*}y-\Rad{D}_{k*}|y|\geq\unum{a}_k,\ 
 k\in K,\\ 
&\umace{C}_{k*}x+\Mid{D}_{k*}y-\Rad{D}_{k*}|y|<\unum{a}_k,\ 
 k\not\in K\}
\end{align*}
either is a subset of $\Ss$ or is disjoint. Since larger $I$ and $K$ preserve optimality, we can remove the strict inequalities, and $\Ss$ is formed by a union of some of the sets 
\begin{align*}
\mna{F}\cap
 \{(x,y)\in\R^{n+n'}\mmid &x_i=0,\ i\in I,\ 
\umace{C}_{k*}x+\Mid{D}_{k*}y-\Rad{D}_{k*}|y|\geq\unum{a}_k,\ k\in K\}.
\end{align*}
There are $2^{m'+n}$ possibilities to choose $I$ and $K$, but by Sperner's theorem again, only at most $\binom{m'}{\lfloor m'/2\rfloor}\binom{n}{\lfloor n/2\rfloor}$ of them are sufficient to consider. 

It remains to prove that the set of feasible solutions fulfilling the active set requirements is a convex polyhedral set. Concerning $I$, the condition $x_i=0$, $i\in I$, is obviously convex preserving. Concerning $K$, the condition
\begin{align*}
\umace{C}_{k*}x+\Mid{D}_{k*}y-\Rad{D}_{k*}|y|\geq \unum{a}_k
\end{align*}
can be reformulated as
\begin{align*}
\umace{C}_{k*}x+\Mid{D}_{k*}y-\Rad{D}_{k*}z\geq \unum{a}_k,\ \ 
z\geq y,\ \ z\geq -y.
\end{align*}
These inequalities describe a convex polyhedral set and its projection to the $x,y$-subspace is also convex polyhedral.
\end{proof}

\section{Conclusion}

We introduced a novel kind of robustness in linear programming. When showing some basic properties, some open questions raised. For example, the robust optimal solution set $\Ss$ may be disconnected, but what can be the number of components at most? Similarly, how tight is the number of convex bodies (Propositions~\ref{propSsEqUniConv} and~\ref{propSsGenUniConv}) the set $\Ss$ consists of?

\subsubsection*{Acknowledgments.} 

The author was supported by the Czech Science Foundation Grant P402/13-10660S.


\bibliographystyle{abbrv}
\bibliography{ae_ilp}

\end{document}